\documentclass[12pt,a4paper,leqno]{article}

\usepackage[centertags]{amsmath}
\usepackage{amsthm}
\usepackage{amssymb,exscale}
\usepackage{latexsym}
\usepackage{hyperref}

\swapnumbers
\theoremstyle{definition}
\newtheorem{theorem}[equation]{Theorem}
\newtheorem{lemma}[equation]{Lemma}
\newtheorem{corollary}[equation]{Corollary}

\newtheorem{proposition}[equation]{Proposition}
\newtheorem{remark}[equation]{Remark}

\renewcommand{\phi}{\varphi}

\newcommand{\D}{\mathrm{d}}
\newcommand{\E}{\mathrm{e}}

\renewcommand{\(}{\bigl(}
\renewcommand{\)}{\bigr)\vphantom{)}}

\newcommand{\impl}{\>\Longrightarrow\>}

\newcommand{\One}{{1\hskip-2.5pt{\rm l}}}

\newcommand{\eps}{\varepsilon}

\newcommand{\cO}{\mathcal O}

\newcommand{\la}{\lambda}

\newcommand{\Ex}{\mathbb E\,}
\newcommand{\R}{\mathbb R}

\renewcommand{\Pr}[1]{\mathbb{P}\mskip1.5mu\(\mskip1.5mu#1\mskip1.5mu\)}

\begin{document}

\title{Expected difference of order statistics\\ in terms of hazard rate}

\author{Boris Tsirelson}

\date{}
\maketitle

\begin{abstract}
If the hazard rate $ \frac{ F'(x) }{ 1-F(x) } $ is increasing (in $x$), then $
\Ex ( X_{n:n} - X_{n-1:n} ) $ is decreasing (in $n$), and moreover, completely
monotone.
\end{abstract}

Motivated by relevance of the expected difference $ \Ex ( X_{n:n} - X_{n-1:n} )
$ of the highest two order statistics to the theory of auctions \cite{RE},
Michael Landsberger asked me \cite{La} about a useful sufficient condition for
this expected difference to be decreasing in the sample size $n$. Here I answer
his question.

Throughout, $ X_1,\dots,X_n $ are independent, identically distributed random
variables, $ \Pr{ X_k \le x } = F(x) $ for all $x$ and $k$; we assume existence
of $x$ such that $ 0<F(x)<1 $ (otherwise the distribution degenerates into a
single atom); $ X_{1:n} \le X_{2:n} \le \dots \le X_{n:n} $ are the
corresponding order statistics; $ \Ex |X_k| < \infty $ (integrability); and
\[
R_n = \Ex ( X_{n:n} - X_{n-1:n} ) \quad \text{for } n \ge 2 \, .\footnotemark
\]
\footnotetext{%
 In fact, $ 0 \le R_n < \infty $ (integrability of $X_k$ implies integrability
 of order statistics, since $ |X_{k:n}| \le \max(|X_1|,\dots,|X_n|) \le
 |X_1|+\dots+|X_n| $).}

\begin{theorem}\label{th}
If the function\footnote{%
 This is a function $ (-\infty,\infty) \to [0,\infty] $.}
$ \phi(x) = -\log \(1-F(x)\) $ is convex, then

(a) the sequence $ (R_n)_n $ is decreasing, that is, $ R_{n+1} \le R_n $ for all
$ n \ge 2 $;

(b) the sequence $ (R_n)_n $ is logarithmically convex, that is, $ R_n \le
\sqrt{ R_{n-1} R_{n+1} } $ for all $ n \ge 3 $;

(c) moreover, the sequence $ (R_n)_n $ is completely monotone.\footnote{%
 About completely monotone sequences see \cite{CM} and \cite[Ch.~III,
 Sect.~4]{Wi}.}
\end{theorem}

The hazard rate (called also failure rate) being the derivative $ \la(x) =
\phi'(x) = \frac{ F'(x) }{ 1-F(x) } $, its increase means convexity of $ \phi
$.\footnote{%
 Convexity of $\phi$ is defined by $ \phi \( \theta a + (1-\theta)b \) \le
 \theta \phi(a) + (1-\theta) \phi(b) $ for all $ \theta \in [0,1] $ and $ a,b
 \in \R $; differentiability is not assumed (nor implied; but local absolute
 continuity is implied). If $\phi$ is continuously differentiable (or at least,
 absolutely continuous), then its convexity is equivalent to increase of the
 hazard rate on $ \{ x : F(x)<1 \} $.}

The inequality (b) implies $ R_n \le \frac12 ( R_{n-1} + R_{n+1} )
$. Equivalently,
\begin{equation}
\frac{R_n}{R_{n+1}} \le \frac{R_{n-1}}{R_n} \quad \text{and} \quad R_n-R_{n+1}
\le R_{n-1}-R_n \, .
\end{equation}

Convexity of $\phi$ is still assumed in two corollaries of (the proof of)
Theorem \ref{th}.

\begin{corollary}\label{cor3}
(a) If the given distribution is (shifted) exponential, that is, of the
form\footnote{%
 This is the distribution of the random variable $ \frac1\la Y + L $ where $Y$
 is standard exponential.}
\[
F(x) = \begin{cases}
 0 &\text{for } x \in (-\infty,L],\\
 1 - \E^{-\la(x-L)} &\text{for } x \in [L,\infty)
\end{cases}
\]
where $ 0<\la<\infty $ and $ -\infty < L < \infty $, then $ R_n = \frac1\la $
for all $ n \ge 2 $.

(b) Otherwise $ R_{n+1} < R_n $ for all $ n \ge 2 $, and $ R_n-R_{n+1} <
R_{n-1}-R_n $ for all $ n \ge 3 $.
\end{corollary}

\begin{corollary}\label{cor4}
(a) If the given distribution is of the form\footnote{%
 This is the distribution of the random variable $ \min \( M, \frac1\la Y + L \)
 $ where $Y$ is standard exponential. Note that the distribution of corollary
 \ref{cor3}(a) is the special case $ M = \infty $.}
\[
F(x) = \begin{cases}
 0 &\text{for } x \in (-\infty,L],\\
 1 - \E^{-\la(x-L)} &\text{for } x \in [L,M),\\
 1 &\text{for } x \in [M,\infty)
\end{cases}
\]
where $ 0<\la<\infty $ and $ -\infty < L < M \le \infty $, then $ R_n =
\frac1\la \( 1 - \E^{-\la(M-L)} \)^n $ for all $ n \ge 2 $.

(b) Otherwise $ \frac{R_n}{R_{n+1}} < \frac{R_{n-1}}{R_n} $ for all $ n \ge 3 $
(and holds the conclusion of Item (b) of Corollary \ref{cor3}).
\end{corollary}

It is widely known that $ \Ex Y = \int_0^\infty \(1-F_Y(x)\) \, \D x =
\int_0^\infty \Pr{Y>x} \, \D x $ whenever $Y$ is a random variable such that $
\Pr{Y\ge0} = 1 $. Here is a slightly more general fact.

\begin{lemma}\label{lemma2}
$ \Ex (Y-Z) = \int_{-\infty}^\infty \Pr{ Z \le x < Y } \, \D x $ whenever $ Y,Z $ are
random variables such that $ \Pr{ Y \ge Z } = 1 $.
\end{lemma}

\begin{proof}
Using the indicator function $ \One_{[a,b)} $ of an interval $ [a,b) $ (equal to
$1$ on $[a,b)$ and $0$ outside), we have $ \int_{-\infty}^\infty \One_{[a,b)}
(x) \, \D x = b-a $; and for the random interval $ [Z,Y) $ we have $ \Ex
\One_{[Z,Y)} (x) = \Pr{ x \in [Z,Y) } = \Pr{ Z \le x < Y } $. Thus,
\[
\Ex (Y-Z) = \Ex \! \int_{-\infty}^\infty \!\! \One_{[Z,Y)} (x) \, \D x =
\int_{-\infty}^\infty \!\! \Ex \One_{[Z,Y)} (x) \, \D x = \int_{-\infty}^\infty
\!\! \Pr{ Z \le x < Y } \, \D x \, .
\]
\end{proof}

\begin{lemma}\label{lemma3}
\[
\Ex (X_{k+1:n}-X_{k:n}) = \binom n k \int_{-\infty}^\infty F^k(x) (1-F(x))^{n-k}
\, \D x
\]
for all $ n=2,3,\dots $ and $ k = 1,\dots,n-1 $; in particular (for $ k = n-1 $),
\begin{equation}\label{eq4}
R_n = n \int_{-\infty}^\infty F^{n-1}(x) (1-F(x)) \, \D x \, .
\end{equation}
\end{lemma}

\begin{proof}\let\qed\relax
Given $x$, the random set of all $ i \in \{1,\dots,n\} $ such that $ X_i \le x $
has $ K_x $ elements, where $ K_x $ is a random variable distributed binomially,
$ \mathrm B (n,F(x)) $. Thus, $ \Pr{ K_x = k } = \binom n k F^k(x)
(1-F(x))^{n-k} $. On the other hand, $ K_x = k $ if and only if $ X_{k:n} \le x
< X_{k+1:n} $. Using Lemma \ref{lemma2},
\begin{multline*}
\Ex (X_{k+1:n}-X_{k:n}) = \! \int_{-\infty}^\infty \! \Pr{ X_{k:n} \le x < X_{k+1:n} }
\, \D x = \! \int_{-\infty}^\infty \! \Pr{ K_x = k } \, \D x = \\
= \int_{-\infty}^\infty \binom n k F^k(x) (1-F(x))^{n-k} \, \D x \,
. \qquad\rlap{$\qedsymbol$}
\end{multline*}
\end{proof}

Proposition \ref{prop5} (below), being a weakened version of Theorem \ref{th},
may suffice a reader not interested in mathematical intricacies. Here we assume
that $ F(0) = 0 $, that is, $ \Pr{ X_n>0 } = 1 $, introduce $ M \in (0,\infty]
$ such that $ F(x)<1 $ if and only if $ x<M $,\footnote{%
 Note two cases allowed, $ M<\infty $ and $ M=\infty $.}
and extend the sequence $ R_2,R_3,\dots $ with one more term
\[
R_1 = \Ex X_1 = \int_0^\infty \( 1 - F(x) \) \, \D x \, ,
\]
which conforms to \eqref{eq4} for $ n=1 $. The case $ M<\infty $ has two
subcases: either $ 0 < F(M-) < 1 $ (atom at $M$), or $ F(M-)=1 $, both covered
by Prop.~\ref{prop5}. In contrast, if $ M=\infty $, then $ F(M-)=1 $; this case
is also covered by Prop.~\ref{prop5}.

\begin{proposition}\label{prop5}
Assume that $ F(0)=0 $, $ F(x)>0 $ for all $ x>0 $, $ F $ is twice
continuously differentiable on $ (0,M) $, and the function $ \phi(x) = -\log
\(1-F(x)\) $ is convex on $(0,M)$. Then $ R_{n+1} \le R_n $ for all $ n \ge 1
$.
\end{proposition}

In the lemma below it is natural and convenient to use instead of $\la(x)$ the
inverse hazard rate\footnote{%
 Not to be confused with the reverse hazard rate $ \frac{F'(x)}{F(x)} $. The
 inverse hazard rate is closely related to the so-called mean time between
 failures (MTBF).}
$ \mu(x) = \frac1{\la(x)} = \frac1{\phi'(x)} = \frac{1-F(x)}{F'(x)} $ for $ x
\in (0,M) $. It is decreasing (since $ \la(x) $ is increasing), this is why
below we prefer the (positive) differential $ \D(-\mu(x)) = -\mu'(x) \, \D x $
of an increasing (and negative) function to the (negative) differential $
\D\mu(x) $ of a decreasing (and positive) function. Note that $ x>0 \impl
F(x)>0 \impl \phi(x)>0 \impl \phi'(x)>0 \impl \mu(x)<\infty $ due to convexity
of $\phi$.

\begin{lemma}\label{lemma6}
In the assumptions of Prop.~\ref{prop5} holds
\[
R_n = \int_0^M F^n(x) \, \D \( -\mu(x) \) + \mu(M-) F^n(M-) \qquad \text{for all
} n \ge 1 \, .
\]
\end{lemma}

\begin{proof}
By \eqref{eq4}, for all $ n \ge 1 $,
\[
R_n = n \! \int_0^M \!\!\!\! F^{n-1}(x) (1-F(x)) \, \D x = n \! \int_0^M \!\!\!\!
F^{n-1}(x) \mu(x) F'(x) \, \D x = \!\! \int_0^M \!\!\!\! \mu(x) \, \D F^n(x)
. \footnotemark
\]
\footnotetext{%
 The integral over $(0,M)$; the possible jump at $M$ is irrelevant.}
We integrate by parts, taking into account that $ \mu(M-)<\infty $:
\begin{multline*}
R_n = \mu(x) F^n(x) \big|_{0+}^{M-} \! - \int_0^M \!\!\! F^n(x)  \, \D \mu(x) =
 \\
= \mu(M-) F^n(M-) - \mu(x) F^n(x) \big|_{0+} \! - \int_0^M \!\!\! F^n(x) \, \D
\mu(x) \, .
\end{multline*}
It remains to prove that $ \mu(x) F^n(x) \big|_{0+} = 0 $. We note that $
\phi(x) \ge F(x) $ (since $ \log q \le q-1 $ for $ 0<q\le1 $), thus $
\mu(x) F^n(x) \le \mu(x) F(x) \le \frac{\phi(x)}{\phi'(x)} $. Convexity of $\phi$
implies $ \phi'(x) \ge \frac{ \phi(x)-\phi(0) }{ x } = \frac{ \phi(x) }{ x } $
for $ x>0 $. Thus, $ \mu(x) F^n(x) \le \phi(x) \cdot \frac{ x }{ \phi(x) } = x
$, therefore $ \mu(x) F^n(x) \big|_{0+} = 0 $.
\end{proof}

\begin{proof}[Proof of Prop.~\ref{prop5}]\let\qed\relax
\begin{multline*}
R_n - R_{n+1} = \int_0^M \underbrace{ \( F^n(x) - F^{n+1}(x) \) }_{\ge0} \,
\underbrace{ \D (-\mu(x) \Big) }_{\ge0} + \\
+ \underbrace{ \mu(M-) }_{\ge0} \cdot \underbrace{ (F^n(M-)-F^{n+1}(M-))
}_{\ge0} \ge 0 \, . \qquad\rlap{$\qedsymbol$}
\end{multline*}
\end{proof}

\begin{remark}\label{remark7}
The inequality $ R_n \le \sqrt{ R_{n-1} R_{n+1} } $, stated in Theorem \ref{th},
can be proved now, for all $ n \ge 2 $, in the assumptions of
Prop.~\ref{prop5}. To this end we use Lemma \ref{lemma6}, Cauchy--Schwarz inequality
for sums, and Cauchy--Schwarz inequality for integrals. By the latter, denoting $ I_n =
\int_0^M F^n(x) \, \D \( -\mu(x) \) $, we have
\begin{multline*}
I_n^2 = \Big( \int_0^M \sqrt{ F^{n-1}(x) F^{n+1}(x) } \, \D \( -\mu(x) \) \Big)^2 \le \\
\le \Big( \int_0^M F^{n-1}(x) \, \D \( -\mu(x) \) \Big)
 \Big( \int_0^M F^{n+1}(x) \, \D \( -\mu(x) \) \bigg) = I_{n-1} I_{n+1} \, .
\end{multline*}
And by the former, denoting $ c = \mu(M-) $, we have
\begin{multline*}
R_n^2 = \( I_n + c F^n(M-) \)^2 \le \( \sqrt{I_{n-1}} \sqrt{I_{n+1}} +
 c \sqrt{ F^{n-1}(M-) } \sqrt{ F^{n+1}(M-) } \)^2 \le \\
\le \( I_{n-1} + c F^{n-1}(M-) \) \( I_{n+1} + c F^{n+1}(M-) \) = R_{n-1}
R_{n+1} \, .
\end{multline*}
\end{remark}

Now we return to the general setup, waiving additional assumptions (that $ M>0
$, $ F(0)=0 $, $ F(x)>0 $ for all $x>0$, and differentiability) while requiring
convexity of the function $ \phi(x) = -\log \( 1-F(x) \) $ on
$(-\infty,\infty)$. Again, $ n \ge 2 $. We define $ L, M $ by
\begin{gather*}
L = \sup \{ x : F(x)=0 \} = \inf \{ x : F(x)>0 \} \, , \\
M = \sup \{ x : F(x)<1 \} = \inf \{ x : F(x)=1 \} \, ,
\end{gather*}
and note that $ -\infty \le L < M \le \infty $, $ F : (L,M] \to (0,1] $ is
increasing, $ F(L) = 0 $ (in the case $ L>-\infty $ convexity of $\phi$ on
$(-\infty,M)$ implies continuity at $L$ of $\phi$ and $F$), and $ F(M) = 1 $
(but $ F(M-) $ may be less if $ M < \infty $). In the case $M=\infty$ we treat $
(L,M] = (L,\infty] $ as an interval on the extended real line $ [-\infty,\infty]
$; in this case $ F(M-) = F(\infty) = 1 $. Accordingly, $ \phi : (L,M) \to
(0,\infty) $ is increasing, $ \phi(L+) = 0 $, and $ \phi(M-) = \infty $ if $ M =
\infty $; otherwise, if $ M < \infty $, $ \phi(M-) $ may be finite or infinite.

The convex function $\phi$ on $(L,M)$ need not be differentiable, but has left
and right derivatives $ \phi'_-, \phi'_+ : (L,M) \to (0,\infty) $.\footnote{%
 They do not vanish, since $ \phi'_-(x) \ge \frac{\phi(x)-\phi(a)}{x-a} $ for
 all $ a \in (L,x) $, and $ \phi(x)-\phi(a) > 0 $ for $a$ close to $L$.}
Both are increasing (and may be unbounded near $M$); $ \phi'_- $ is left
continuous, $ \phi'_+ $ is right continuous; $ \phi'_- \le \phi'_+ $; and the
set $ \{ x \in (L,M) : \phi'_-(x) < \phi'_+(x) \} $ is at most countable. We
define $ \mu : (L,M] \to [0,\infty) $ by
\begin{align*}
& \mu(x) = \frac1{ \phi'_+(x) } \qquad \text{for } L<x<M \, , \\
& \mu(M) = 0 \, ,
\end{align*}
and observe that $ \mu $ is decreasing and right continuous (and may be
unbounded near $L$). The corresponding (positive, locally finite) Stieltjes
measure $ \D(-\mu) $ on $(L,M]$ is defined by $ \int_{(a,b]} \D(-\mu) =
\mu(a)-\mu(b) $ whenever $ L<a<b\le M $. This measure has atoms at the points of
discontinuity of $ \mu $ (if any), that is, the points of $ \{ x \in (L,M) :
\phi'_-(x) < \phi'_+(x) \} $, and in addition, at $M$, if $ \mu(M-)>0 $ (which
may happen, be $M$ finite or infinite).

In order to prove Theorem 1 we first generalize Lemma \ref{lemma6}.

\begin{lemma}\label{lemma8}
\[
R_n = \int_{(L,M]} F^n(x-) \, \D \( -\mu(x) \) \qquad \text{for all } n \ge 2
\]
(Lebesgue-Stieltjes integral).
\end{lemma}

\begin{proof}
We use integration by parts for Lebesgue-Stieltjes integrals \cite{LSI,He}:
\begin{equation}\label{eq9}
\int_{(a,M)} \!\!\!\! \mu(x) \, \D F^n(x) = 
\mu(M-) F^n(M-) - \mu(a+) F^n(a+) - \int_{(a,M)} \!\!\!\!\! F^n(x) \, \D \mu(x)
\end{equation}
for every $ a \in (L,M) $, since on $(a,M)$ both functions, $\mu$ and $F^n$, are
bounded, monotone, $\mu$ is right continuous, and $F^n$ is continuous. (On the
left-hand side, values of the integrand $\mu(x)$ at points of discontinuity do
not matter, since the integrator $F^n(x)$ is continuous.)

In order to take the limit $ a \to L+ $ we note that, by convexity of $ \phi $,
$ \phi(x) \ge \phi(a) - (a-x) \phi'(a-) $ for all $ x \in (-\infty,a) $, and $
F(x) = 1 - \E^{-\phi(x)} \ge \frac{\phi(x)}{\phi(a)} ( 1 - \E^{-\phi(a)} ) $ by
convexity of $ \exp $.
For $ a \to L+ $ we have $ \phi(a) \to 0 $, thus $
\frac{1-\E^{-\phi(a)}}{\phi(a)} \to 1 $. We take $ a_0 \in (L,M) $ such that $ 
\frac{1-\E^{-\phi(x)}}{\phi(x)} \ge \frac12 $ for all $ x \in (L,a_0) $; then $
F(x) \ge \frac12 \phi(x) $ for all $ x \in (-\infty,a_0) $. Thus, $ F(x) \ge
\frac12 \( \phi(a) - (a-x) \phi'(a-) \) $ whenever $ -\infty < x < a < a_0 $, $
a > L $. Therefore
\[
\int\limits_{-\infty}^a \! F(x) \, \D x \ge \frac12 \!\!
\int\limits_{a-\frac{\phi(a)}{\phi'(a-)}}^a \!\!\! \(
\phi(a) - (a-x) \phi'(a-) \) \, \D x = \frac14 \phi(a) \cdot
\frac{\phi(a)}{\phi'(a-)} \ge \frac14 \mu(a) \phi^2(a)
\]
for all $ a \in (L,a_0) $. Taking into account that $ 0 \le \int_{-\infty}^a
F(x) \, \D x = \Ex \max(0,a-X_k) \to 0 $ as $ a \to L+ $ by integrability of $
X_k $, and $ F(x) \le \phi(x) $ everywhere, we get
\begin{equation}\label{eq10}
\mu(a) F^n(a) \to 0 \quad\text{as } a \to L+
\end{equation}
for each $ n \ge 2 $.

For every $ b \in (L,M) $, on $(L,b)$ the bounded convex function $\phi$, being
Lipschitz, is absolutely continuous, which implies absolute continuity of
$F$ and $F^n$ on $ (L,b) $, therefore, on $(L,M)$ (since the variation on
$[b,M)$ converges to $0$ as $ b \to M- $). This way a Lebesgue-Stieltjes
integral against $\D F^n(x)$ turns into a Lebesgue integral with $ \(F^n(x)\)'
\, \D x $:
\begin{multline*}
 \int_{(a,M)} \mu(x) \, \D F^n(x) =  \int_a^M \frac{ 1-F(x) }{ F_+'(x) }
 n F^{n-1}(x) F'(x) \, \D x = \\
= n \int_a^M F^{n-1}(x)(1-F(x)) \, \D x \to n \int_L^M F^{n-1}(x)(1-F(x)) \, \D
 x
\end{multline*}
as $ a \to L+ $. (Values of $F'$ at points of discontinuity do not matter.)
Using \eqref{eq9} and \eqref{eq10},
\[
\int_{(a,M)} F^n(x) \, \D \mu(x) \> \to \> \mu(M-) F^n(M-) - n \int_L^M
F^{n-1}(x)(1-F(x)) \, \D x
\]
as $ a \to L+ $.
Taking into account that $ \D \( -\mu(\cdot) \) $ is a
well-defined positive, locally finite measure on $ (L,M) $ (even if $
\mu(L+)=\infty $), we conclude that the integral $ \int_{(L,M)} F^n(x) \, \D
\(-\mu(x)\) $ is well-defined, and
\[
\int_{(L,M)} F^n(x) \, \D \(-\mu(x)\) = n \int_L^M F^{n-1}(x)(1-F(x))
\, \D x - \mu(M-) F^n(M-) \, .
\]
Finally, taking into account that $ \mu(M)=0 $ we get
\begin{multline*}
\int_{(L,M]} F^n(x-) \, \D \( -\mu(x) \) = \\
= \int_{(L,M)} F^n(x-) \, \D \( -\mu(x) \) + F^n(M-) \( -\mu(M) + \mu(M-) \) = \\
= \int_{(L,M)} F^n(x) \, \D \( -\mu(x) \) + \mu(M-) F^n(M-) = n \int_L^M
 F^{n-1}(x)(1-F(x)) \, \D x = \\
= n \int_{-\infty}^\infty F^{n-1}(x)(1-F(x)) \, \D x = R_n
\end{multline*}
by \eqref{eq4}.
\end{proof}

\begin{proof}[Proof of Theorem \ref{th}]
For each $ n \ge 2 $, by Lemma \ref{lemma8}, $ R_n = \int_{(L,M]} F^n(x-)
\, \D \( -\mu(x) \) $. The change of variable, $ p = F(x-) $, gives
\[
R_n = \int_{(0,1]} p^n \, \D \( -\nu_1(p) \) \, ,
\]
where $ \nu_1 : (0,1] \to [0,\infty) $ is the decreasing right continuous
function defined by
\begin{align*}
& \nu_1 \( F(x) \) = \mu(x) && \text{for all } x \in (L,M) \, , \\
& \nu_1(p) = 0 && \text{for all } p \in [F(M-),1] \, .
\end{align*}
A second change of variable, $ t = -\log p $ (that is, $ p = \E^{-t} $), gives
\[
R_n = \int_{[0,\infty)} \E^{-nt} \, \D \nu_2(t) \, ,
\]
where $ \nu_2 : [0,\infty) \to [0,\infty) $ is the increasing left continuous
function defined by $ \nu_2(-\log p) = \nu_1(p) $ for all $ p \in (0,1] $
(that is, $ \nu_2(t) = \nu_1 (\E^{-t}) $ for all $ t \in [0,\infty) $), and the
corresponding (positive, locally finite) Stieltjes measure $ \D\nu_2 $ on
$[0,\infty)$ is defined by $ \int_{[a,b)} \D\nu_2 = \nu_2(b)-\nu_2(a) $ whenever
$ 0\le a<b<\infty $.
We have $ R_n = R(n) $ where
\[
R(u) = \int_{[0,\infty)} \E^{-ut} \, \D \nu_2(t) \, ,
\]
$ R : [2,\infty) \to [0,\infty) $. This function $R$ is the Laplace transform of
a (positive, locally finite) measure, and is finite on $ [2,\infty) $. Thus, the
function $R$ is completely monotone on $[2,\infty)$ \cite[Ch.~IV, Def.~2(a,b,c)
and Th.~12a]{Wi}, and therefore the sequence $ \( R(2),R(3),\dots \) =
(R_2,R_3,\dots) $ is completely monotone \cite[Ch.~III, Def.~4 and Ch.~IV,
Th.~11d]{Wi}, which proves Item (c) of Theorem \ref{th}. Item (a) follows
immediately. Item (b) follows due to the fact that every completely monotone
function is logarithmically convex \cite[Ch.~IV, Th.~16 and Corollary
16]{Wi}.\footnote{%
 Alternatively, use Cauchy--Schwarz inequality (for integrals), as in Remark
 \ref{remark7}.}
\end{proof}

\begin{proof}[Proof of Corollary \ref{cor3}]
(a) $ \phi(x) = -\log \( 1 - (1-\E^{-\la(x-L)}) \) = \la(x-L) $ for all $ x \in
[L,\infty) $, thus $ \phi'(x) = \la $, $ \mu(x) = \frac1\la $, and Lemma
\ref{lemma8} gives $ R_n = \mu(\infty) = \frac1\la $.

(b) Assume toward contradiction that $R$ is constant on $[2,\infty)$, then: $
\D\nu_2 $ is a single atom at $0$; $ \nu_1 $ is constant on $(0,1)$; $ \mu $ is
constant on $(L,M)=(L,\infty)$; $ \phi'_+ $ is constant on $(L,\infty)$; $\phi$
is linear on $(L,\infty)$, which cannot happen when the given distribution is
not shifted exponential. Therefore $R$ is not constant on $[2,\infty)$.

Being analytic on $(2,\infty)$ \cite[Ch.~IV, Sect.~3]{Wi}, $R$ is not constant on $(a,b)$
whenever $ 2 \le a < b < \infty $. Thus $R$ is strictly decreasing on
$[2,\infty)$, whence $ R_{n+1}<R_n $ for all $n\ge2$.
\end{proof}

\begin{proof}[Proof of Corollary \ref{cor4}]
(a) $ \phi(x) = \la(x-L) $ for $ x \in (L,M) $, thus $ \mu(x) = \frac1\la $ for
$ x \in (L,M) $, and $ \mu(M)=0 $; $ \D\mu $ is a single atom at $M$; Lemma
\ref{lemma8} gives $ R_n = \int_{(L,M]} F^n(x-) \, \D\(-\mu(x)\) = F^n(M-)
\mu(M-) = ( 1 - \E^{-\la(M-L)} )^n \cdot \frac1\la $.

(b) First, $F$ is not constant (recall the proof of Corollary \ref{cor3}, and
consider $ M = \infty $). Assume toward contradiction that $R$ is exponential
on $(2,\infty)$, that is, $ R(u) = c \E^{-au} $ for all $ u \in (2,\infty) $
where $ a,c>0 $; then: $ \D\nu_2 $ is a single atom at $ t_0 = a \in (0,\infty)
$; $ \D\nu_1 $ is a single atom at $ p_0 = \E^{-t_0} \in (0,1) $, and $
\nu_1(1)=0 $.

On one hand, $ \nu_1(\cdot) = c>0 $ on $(0,p_0) $ and $ \nu_1(\cdot) = 0 $ on $
[p_0,1] $.

On the other hand, by the definition of $ \nu_1 $, taking into account that $
\mu(\cdot) = \frac1{\phi'_+(\cdot)} > 0 $ on $(L,M)$, we have $ \nu_1(\cdot) >0
$ on $\(0,F(M-)\) $ and $ \nu_1(\cdot) = 0 $ on $ [F(M-),1] $.

Thus, $ F(M-) = p_0 $, and $ \mu(\cdot) = c $ on $(L,M)$, whence $ \phi(x) =
\frac1c (x-L) $ for $x\in(L,M)$, which cannot happen when the given distribution
is not in the form of Item (a). Therefore $R$ is not exponential.

The function $R$ is logarithmically convex (recall the proof of Theorem
\ref{th}) and analytic (recall the proof of Corollary \ref{cor3}). Accordingly,
the function $ \log R(\cdot) $ is convex, analytic, and not linear (since $R$ is
not exponential). Moreover, is not linear on $(a,b)$ whenever $ 2 \le a < b <
\infty $ (since its second derivative cannot vanish on $(a,b)$). Thus $\log
R(\cdot)$ is strictly convex on $[2,\infty)$, whence $ \log R_n - \log R_{n+1} $
is strictly decreasing in $n$.
\end{proof}

\begin{remark}
If the hazard rate is not increasing, then the function $R$ is not completely
monotone on $[2,\infty)$, but it still may happen that the sequence $(R_n)_n$ is
decreasing.
\end{remark}

\begin{remark}
Under some conditions,
\[
R_n \approx \frac1{\la(x_n)} \, ,
\]
where $ x_n $ is the $\frac{n-1}n$-quantile, that is, $ 1-F(x_n) = \frac1n $.

Namely, if $ \frac1{\la(\cdot)} $ is bounded and
\[
\bigg| \frac1{\la(x)} - \frac1{\la(x_n)} \bigg| \le \eps \quad \text{for all } x
\text{ such that } \frac1n \eps \le 1-F(x) \le \frac1n \log \frac1\eps \, ,
\]
then $ \big| R_n - \frac1{\la(x_n)} \big| = \cO(\eps) $.

Thus, if the inverse hazard rate oscillates like, say, $ \frac1{\la(x)} = 2 +
\cos \( \eps \log (1-F(x) \) $ where $\eps>0$ is small, then the sequence $
(R_n)_n $ oscillates accordingly, $ R_n \approx 2 + \cos ( \eps \log n ) $.
\end{remark}

\bigskip
\filbreak
{
\small
\begin{sc}
\parindent=0pt\baselineskip=12pt
\parbox{4in}{
Boris Tsirelson\\
School of Mathematics\\
Tel Aviv University\\
Tel Aviv 69978, Israel
\smallskip
\par\quad\href{mailto:tsirel@post.tau.ac.il}{\tt
 mailto:tsirel@post.tau.ac.il}
\par\quad\href{http://www.tau.ac.il/~tsirel/}{\tt
 http://www.tau.ac.il/\textasciitilde tsirel/}
}

\end{sc}
}
\filbreak

\end{document}